\documentclass[twoside,reqno,12pt]{amsart}
\usepackage[active]{srcltx}
\usepackage{xcolor}
\usepackage{latexsym,rotate,eucal,cite}
\usepackage{amsmath,amsthm,amssymb,amsxtra}
\newtheorem{dl}{Theorem}[section]

\theoremstyle{definition}

\newtheorem{tl}[dl]{Corollary}
\newtheorem{lz}[dl]{Example}

\newtheorem{prop}[dl]{Proposition}

\numberwithin{equation}{section}

\def\qed{\hfill \rule{4pt}{7pt}}
\def\pf{\noindent {\it Proof.} }

\usepackage[hidelinks]{hyperref}

\begin{document} 
\title{Notices on a new  functional identity and allied series transformations}
\author{Jianan Xu*}
\address{Department of Mathematics,
Taizhou University,
Taizhou 225300,~P.~R.~China}
\curraddr{}
\email{xujn1205@163.com} 
\author{Qi Chen}
\address{
Department of Mathematics,
Soochow University,
Suzhou 215006,~P.~R.~China}
\curraddr{}
\curraddr{}
\email{20254007010@stu.suda.edu.cn}
\thanks{*Corresponding author.}

\begin{abstract}
In this paper, we establish an elementary functional identity
\begin{align*}
  \sum_{k=1}^{n}\bigg(\frac{f(a,d_k)}{f(a,x)}\bigg)^{n-1}\prod_{i=1,i \neq k}^{n} \frac{g(x,d_i)}{g(d_k,d_i)}=1
\end{align*}
based on a pair of functions $f(x, y)$ and $g(x,y)$ satisfying
$$
f(x, a) g(b, c)+f(x, b) g(c, a)+f(x, c) g(a, b)=0.
$$
The latter may serve as a prototype for the classical Weierstrass theta identity. As a direct application, we establish a general  transformation formula for  finite sums. Some special algebraic, triangular and elliptic series transformations  are also presented.

\vspace{6pt}

\vspace{6pt}
\end{abstract}

\maketitle\thispagestyle{empty}
\markboth{J. N. Xu and Q. Chen}{Notices on a new  functional identity and allied series transformations}


\parskip 7pt
\baselineskip 16pt



\section{Introduction}
Throughout this paper, we will follow
the notation and terminology in the book \cite{10} by Gasper and Rahman. The $q$-shifted factorials of complex variable $z$ with the base $q$, $|q|<1$, are given by
\begin{align*}
(z;q)_\infty
:=\prod_{n=0}^{\infty}(1-zq^n),\quad (z;q)_n
:=\frac{(z;q)_\infty}{(zq^n;q)_\infty}.
\end{align*} Jacobi's (modified) theta function (see \cite[Eq. (11.2.1)]{10})
is defined to be
\begin{align}
\theta(z;q):=(z,q/z;q)_\infty\quad \mbox{for}\quad z\neq 0.\label{1}
\end{align}
In addition, we often employ the following multi-parameter compact forms for any  complex variable $n$
\begin{align}
  (a_1,a_2,\ldots,a_n;q)_\infty &:=(a_1;q)_\infty (a_2;q)_\infty\cdots (a_n;q)_\infty,\nonumber\\
  \theta(a_1,a_2,\ldots,a_n;q)&:=\theta(a_1;q)\theta(a_2;q)\cdots\theta(a_n;q)\nonumber
\end{align}
and the $p,q$-shifted factorials
\begin{align}
  (a_1,a_2,\ldots,a_n;p,q)_m&:=\prod_{i=0}^{m-1}\theta(a_1p^i;q)\prod_{i=0}^{m-1}\theta(a_2p^i;q)\cdots\prod_{i=0}^{m-1}\theta(a_np^i;q).\label{ellptic-shift}
\end{align}

The classical  Weierstrass theta identity \cite[Ex.~2.16]{10} may be recorded as the following function identity with four non-zero parameters:
\begin{align}
\theta\left(cx,\frac{x}{c},bz,
\frac{z}{b};q\right)-\theta \left(bx,\frac{x}{b},cz,\frac{z}{c};q\right)=\frac{z}{c}\theta \left(bc,\frac{c}{b},
xz,\frac{x}{z};q\right).\label{weierstrass}
\end{align}
As of today,  Weierstrass' theta identity \eqref{weierstrass} has been proved to be very important  in the study of modular, elliptic and theta hypergeometric series. For this, we refer the reader to  \cite[Chap.~11]{10} and \cite{gasper,warnaar,koornwinder} for further details.  

In the previous paper \cite{xuxrma}, we have already  established two elementary functional identities closely related to a pair of functions satisfying
\begin{align} f(x,a)g(b,c)+f(x,b)g(c,a)+f(x,c)g(a,b)=0.\label{jacobi} \end{align}
 One  functional identity  is
\begin{dl}[Cf.{\rm \cite[Thm.~1]{xuxrma}}]Suppose that $g\perp g$ and $g(x,y)=-g(y,x)$. Then, for any pairwise distinct sequence $\{d_n\}_{n\geq 0}$, we have
\begin{align}
   \sum_{k=1}^{n}\bigg(\frac{g(a,d_k)}{g(a,x)}\bigg)^{n-1}\prod_{i=1, i \neq k}^{n} \frac{g(x,d_i)}{g(d_k,d_i)}=1.\label{bmms-one}
\end{align}
\end{dl}
Hereafter, we use $f \perp g$ for  such a pair of functions  subject to \eqref{jacobi} while \eqref{jacobi} is called Jacobi's identity. 
We refer the reader to \cite{fgsummation,wangjinpaper,xuxrma} for its applications to series transformations. In this paper, we will
 extend \eqref{weierstrass}-\eqref{bmms-one} to the following

 \begin{dl}\label{mainthm-before} Suppose that $f\perp g$, $g(x,y)=-g(y,x)$. Then, for any pairwise distinct sequence $\{d_n\}_{n\geq 0}$, we have
 \begin{align}
   \sum_{k=1}^{n}\bigg(\frac{f(a,d_k)}{f(a,x)}\bigg)^{n-1}\prod_{i=1, i \neq k}^{n} \frac{g(x,d_i)}{g(d_k,d_i)}=1.\label{www-www-www-333}
\end{align} 
\end{dl}
The proof of this theorem will be presented in the next section. In Section \ref{sec3}, we will establish  the following new transformation which builds on  the  special case $n=3$ of \eqref{www-www-www-333}.

\begin{dl}\label{mainthm-two}Suppose that $f\perp g$ and $g(x,y)=-g(y,x)$. For any integer $n$ and pairwise distinct sequence $\{x_n,a_n,b_n,c_n,d_n\}_{n\geq 0}$, we have
\begin{align}
 T_1(n)+T_2(n)=U_0-U_n,   
\end{align}
where
\begin{subequations}
\begin{align}
T_1(n)&:=\sum^{n}_{k=1}\frac{f(x_k,d_k)^2g(a_k,b_k)g(a_k,c_k)}{f(x_k,a_k)^{2}
g(d_k,b_k)g(d_k,c_k)}U_{k-1},\label{summand-0}\\
T_2(n)&:=\sum^{n}_{k=1}\frac{f(x_k,c_k)^{2}g(a_k,b_k)g(a_k,d_k)}{f(x_k,a_k)^{2}
g(c_k,b_k)g(c_k,d_k)}U_{k-1}\label{summand-1}
\end{align}
and
\begin{align}
U_n:=\prod_{k=0}^n\frac{f(x_k,b_k)^{2}
  g(a_k,c_k)g(a_k,d_k)}{f(x_k,a_k)^{2}
g(b_k,c_k)g(b_k,d_k)}.\label{summand-2}
\end{align}  
\end{subequations}
\end{dl}
Some concrete and useful cases covered by this transformation are presented.

\section{The proof of Theorem \ref{mainthm-before}}\label{sec2}
This section is devoted to  the proof of Theorem \ref{mainthm-before}. 

\begin{proof}  We proceed to show  \eqref{www-www-www-333} by induction on $n$. For this purpose, we first restate it as follows:
\begin{align}
f(a,x)^{n-1}=\sum_{k=1}^{n}f(a,d_k)^{n-1}\prod_{i=1, i \neq k}^{n} \frac{g(x,d_i)}{g(d_k,d_i)}.\label{chazhiji-newnew-ooo-mmm-1}
\end{align}  
First, when $n=2$, \eqref{chazhiji-newnew-ooo-mmm-1}  turns out to be
    \begin{align*}
    f(a,x)=f(a,d_1)\frac{g(x,d_2)}{g(d_1,d_2)}+f(a,d_2)\frac{g(x,d_1)}{g(d_2,d_1)}.
    \end{align*}
    It is obviously equivalent to \eqref{jacobi}. Assume that \eqref{chazhiji-newnew-ooo-mmm-1} is valid for $n=m$, which says 
    \begin{align}
f(a,x)^{m-1}=\sum_{k=1}^{m}f(a,d_k)^{m-1}\prod_{i=1, i \neq k}^{m} \frac{g(x,d_i)}{g(d_k,d_i)}.\label{chazhiji-newnew-ooo-mmm-2}
\end{align}
    Next, consider the case $n=m+1$. It follows from the induction hypothesis \eqref{chazhiji-newnew-ooo-mmm-2} that
    \begin{align*}
f(a,x)^{m}&=f(a,x)~ f(a,x)^{m-1}\\
&=f(a,x)\sum_{k=1}^{m}f(a,d_k)^{m-1}\prod_{i=1, i \neq k}^{m} \frac{g(x,d_i)}{g(d_k,d_i)}\\
&=  \sum_{k=1}^{m}f(a,x)g(d_k,d_{m+1})\frac{f(a,d_k)^{m-1}}{g(x,d_{m+1})}\prod_{i=1, i \neq k}^{m+1} \frac{g(x,d_i)}{g(d_k,d_i)}.
\end{align*}
Since, by  Jacobi's identity \eqref{jacobi},
$$f(a,x)g(d_k,d_{m+1})=f(a,d_k)g(x,d_{m+1})-f(a,d_{m+1})g(x,d_k),$$
thus
 \begin{align*}
f(a,x)^{m}&= \sum_{k=1}^{m}f(a,d_k)g(x,d_{m+1})\frac{f(a,d_k)^{m-1}}{g(x,d_{m+1})}\prod_{i=1, i \neq k}^{m+1} \frac{g(x,d_i)}{g(d_k,d_i)} \\
&-\sum_{k=1}^{m} f(a,d_{m+1})g(x,d_k)\frac{f(a,d_k)^{m-1}}{g(x,d_{m+1})}\prod_{i=1, i \neq k}^{m+1} \frac{g(x,d_i)}{g(d_k,d_i)}\\
&=\sum_{k=1}^{m}f(a,d_k)^{m}\prod_{i=1, i \neq k}^{m+1} \frac{g(x,d_i)}{g(d_k,d_i)}+R_m
\end{align*}
All remains to show
 \begin{align*}
 R_m&:=-
f(a,d_{m+1})\prod_{i=1}^{m+1}g(x,d_i)\sum_{k=1}^{m}\frac{f(a,d_k)^{m-1}}{g(x,d_{m+1})}\prod_{i=1, i \neq k}^{m+1} \frac{1}{g(d_k,d_i)}\\
&=f(a,d_{m+1})^{m}\prod_{i=1}^{m} \frac{g(x,d_i)}{g(d_{m+1},d_i)}.
\end{align*}
After a bit simplification, it is equivalent to
 \begin{align*}
\sum_{k=1}^{m}f(a,d_k)^{m-1}\prod_{i=1, i \neq k}^{m} \frac{g(d_{m+1},d_i)}{g(d_k,d_i)}=f(a,d_{m+1})^{m-1}.
\end{align*}
It is just the special case $x=d_{m+1}$ of \eqref{chazhiji-newnew-ooo-mmm-2}. Thus, \eqref{chazhiji-newnew-ooo-mmm-1} holds for $n=m+1$. By induction, we conclude that \eqref{chazhiji-newnew-ooo-mmm-1}, thus \eqref{www-www-www-333}, is  true for all integers $n\geq 2$.
 \end{proof}
 We emphasize here that \eqref{www-www-www-333} is different from \eqref{bmms-one}  in that the former is a generalization of Jacobi's identity \eqref{jacobi}. It is just this transformation with which we are able to set up a transformation of finite sums. To that end,  we need  the following special case of \eqref{www-www-www-333}. 
\begin{tl}[$x=d_{n+1}$]For any pairwise distinct sequence $\{d_n\}_{n\geq 0}$,we have
    \begin{align}
   \sum_{k=1}^n\bigg(\frac{f(a,d_k)}{f(a,d_{n+1})}\bigg)^{n-1}
   \prod_{i=1,i\neq k}^{n}\frac{g(d_{n+1},d_i)}{g(d_k,d_i)} =1.\label{www-www-www-3333}
\end{align} 
\end{tl}
\begin{prop}[$n=2$:  Jacobi's identity] 
\begin{align}
f(a,d_1)g(d_{2},d_3)+f(a,d_2)g(d_{3},d_1)+f(a,d_{3})g(d_{1},d_2)=0.\label{www-www-www-444}
\end{align}
\end{prop}
\pf In this case, \eqref{www-www-www-3333} reduces to
\begin{align*}
   \frac{f(a,d_1)}{f(a,d_3)}
   \frac{g(d_3,d_2)}{g(d_1,d_2)}+\frac{f(a,d_2)}{f(a,d_3)}
  \frac{g(d_{3},d_1)}{g(d_2,d_1)} = 1,
\end{align*}  
which coincides with  Jacobi's identity \eqref{www-www-www-444}.
\qed
\begin{prop}[$n=3$] 
\begin{align}
&  f(a,d_1)^{2}
g(d_2,d_3)g(d_{2},d_4)g(d_{3},d_4)
-f(a,d_2)^{2}
  g(d_1,d_{3})g(d_1,d_4)g(d_3,d_{4})\nonumber\\
&+f(a,d_3)^{2}g(d_1,d_2)g(d_1,d_{4})g(d_2,d_{4})
=f(a,d_{4})^2g(d_1,d_2)g(d_1,d_3)g(d_2,d_3).\label{www-www-www-www-444}
\end{align}
\end{prop}
\pf In this case,
 \eqref{www-www-www-3333} corresponds to
\begin{align*}
   \bigg(\frac{f(a,d_1)}{f(a,d_{4})}\bigg)^{2}
   \frac{g(d_{4},d_2)}{g(d_1,d_2)}\frac{g(d_{4},d_3)}{g(d_1,d_3)}+\bigg(\frac{f(a,d_2)}{f(a,d_{4})}\bigg)^{2}
  \frac{g(d_{4},d_1)}{g(d_2,d_1)} \frac{g(d_{4},d_3)}{g(d_2,d_3)}\\
  +\bigg(\frac{f(a,d_3)}{f(a,d_{4})}\bigg)^{2}\frac{g(d_{4},d_1)}{g(d_3,d_1)}
   \frac{g(d_{4},d_2)}{g(d_3,d_2)} =1,
\end{align*}
which turns out to be \eqref{www-www-www-www-444}.
\qed

\section{An allied series transformation}\label{sec3}

To proceed, we first show Theorem \ref{mainthm-two}.  Actually, this transformation belongs to  a direct application of \eqref{www-www-www-www-444} to finite sums.

\pf  Observe that \eqref{www-www-www-www-444} can be reformulated as
\begin{align}
&  f(a,d_1)^{2}
g(d_2,d_3)g(d_2,d_{4})
-f(a,d_2)^{2}
  g(d_1,d_3)g(d_1,d_{4})\nonumber\\
&
=\frac{g(d_1,d_2)}{g(d_3,d_{4})}(f(a,d_{4})^2g(d_1,d_3)g(d_2,d_3)-f(a,d_3)^{2}g(d_1,d_{4})g(d_2,d_{4})),\label{www-www-www-555}
\end{align}
where $a,d_1,d_2,d_3,d_4$ are arbitrary but pairwise distinct complex numbers. 
Now we consider the finite sum
\begin{align}
\sum^{n}_{k=1}(U_{k-1}-U_k)=U_0-U_n,\label{sumsum}
\end{align}
where
\begin{align}
U_n:=\prod_{k=0}^n\frac{f(x_k,b_k)^{2}
  g(a_k,c_k)g(a_k,d_k)}{f(x_k,a_k)^{2}
g(b_k,c_k)g(b_k,d_k)}.\label{summand}
\end{align}
Subsequently, it is without any difficulty to  check 
\begin{align*}
&\mbox{LHS of \eqref{sumsum}}=\sum^{n}_{k=1}\bigg(1-\frac{U_k}{U_{k-1}}\bigg)U_{k-1}\\
&=\sum^{n}_{k=1}\bigg\{1-\frac{f(x_k,b_k)^{2}
  g(a_k,c_k)g(a_k,d_k)}{f(x_k,a_k)^{2}
g(b_k,c_k)g(b_k,d_k)}\bigg\}U_{k-1}\\
&=\sum^{n}_{k=1}\frac{f(x_k,a_k)^{2}
g(b_k,c_k)g(b_k,d_k)-f(x_k,b_k)^{2}
  g(a_k,c_k)g(a_k,d_k)}{f(x_k,a_k)^{2}
g(b_k,c_k)g(b_k,d_k)}U_{k-1}.
\end{align*}
Next, by applying \eqref{www-www-www-555} to the numerator  and simplifying, we come up with
\begin{align*}
&\mbox{LHS of \eqref{sumsum}}\\&=\sum^{n}_{k=1}\frac{f(x_k,d_k)^2g(a_k,c_k)g(b_k,c_k)-f(x_k,c_k)^{2}g(a_k,d_k)g(b_k,d_k)}{f(x_k,a_k)^{2}
g(b_k,c_k)g(b_k,d_k)}\frac{g(a_k,b_k)}{g(c_k,d_k)}U_{k-1}\\
&=T_1(n)+T_2(n),
\end{align*}
where $T_1(n)$ and $T_2(n)$ are given by \eqref{summand-0} and \eqref{summand-1},  respectively.
The theorem is thereby proved.
\qed

As a matter of fact, we are able to set up more general transformation via the same argument.
\begin{dl}\label{mainthm-three}Suppose that $f\perp g$, $g(x,y)=-g(y,x)$ and $\{U_n\}_{n\geq 0}$ is given by \eqref{summand}. Then, for any integer $n$ and arbitrary sequence $\{V_n\}_{n\geq 0}$, as well as pairwise distinct sequence $\{x_n,a_n,b_n,c_n,d_n\}_{n\geq 0}$, we have
\begin{align}
 S_1(n)+S_2(n)=U_{0}V_1-U_{n}V_n+\sum_{k=1}^{n-1} U_k(V_{k+1}-V_{k}),   
\end{align}
where
\begin{subequations}
\begin{align}
S_1(n)&:=\sum^{n}_{k=1}\frac{f(x_k,d_k)^2g(a_k,b_k)g(a_k,c_k)}{f(x_k,a_k)^{2}
g(d_k,b_k)g(d_k,c_k)}U_{k-1}V_k,\label{summand-000}\\
S_2(n)&:=\sum^{n}_{k=1}\frac{f(x_k,c_k)^{2}g(a_k,b_k)g(a_k,d_k)}{f(x_k,a_k)^{2}
g(c_k,b_k)g(c_k,d_k)}U_{k-1}V_k.\label{summand-100}
\end{align}
\end{subequations}
\end{dl}
\pf The consequence follows from Abel's lemma on  summations by parts (cf. \cite[p.~313]{knopp}), namely
\begin{align*}
  \sum_{k=1}^{n}(U_{k-1}-U_{k}) V_k=U_{0}V_1-U_{n}V_n+\sum_{k=1}^{n-1} U_k(V_{k+1}-V_{k}).
\end{align*}
\qed

At the end of this paper, we consider a few concrete algebraic and triangular function identities, including the    transformations for basic/elliptic hypergeometric series covered  Theorem \ref{mainthm-two}. As for applications of Theorem \ref{mainthm-three}, we postpone for forthcoming study. In addition,  all verifications of $f\perp g$ for the relevant $f(x, y)$ and $g(x, y)$ is left to the reader but only referring to \cite{fgsummation} for details.
\begin{tl}[$f(x,y)=g(x,y)=x-y$]
       \begin{align}
    &\sum^{n}_{k=1}\frac{(x_k-d_k)^2(a_k-b_k)(a_k-c_k)}{(x_k-a_k)^{2}
(d_k-b_k)(d_k-c_k)}\prod_{i=0}^{k-1}\frac{(x_i-b_i)^{2}
  (a_i-c_i)(a_i-d_i)}{(x_i-a_i)^{2}
(b_i-c_i)(b_i-d_i)}\nonumber\\
&+\sum^{n}_{k=1}\frac{(x_k-c_k)^{2}(a_k-b_k)(a_k-d_k)}{(x_k-a_k)^{2}
(c_k-b_k)(c_k-d_k)}\prod_{i=0}^{k-1}\frac{(x_i-b_i)^{2}
  (a_i-c_i)(a_i-d_i)}{(x_i-a_i)^{2}
(b_i-c_i)(b_i-d_i)}\label{idnew-1}\\
&=\frac{(x_0-b_0)^{2}
  (a_0-c_0)(a_0-d_0)}{(x_0-a_0)^{2}
(b_0-c_0)(b_0-d_0)}-\prod_{k=0}^n\frac{(x_k-b_k)^{2}
  (a_k-c_k)(a_k-d_k)}{(x_k-a_k)^{2}
(b_k-c_k)(b_k-d_k)}.\nonumber
\end{align} 
\end{tl}
As far as $q$-series is concerned, we set in \eqref{idnew-1} that $a_i=a
,b_i=b (a\neq b)$ and $$x_i=abxq^i,c_i=abcq^i,d_i=abdq^i$$ and establish a new transformation which  seems to have been unknown in the literature.
\begin{prop}For any integer $n\geq 0$, we have 
\begin{align*}
  &\sum^{n}_{k=1}\frac{(ax;q)_{k}^{2}
  (bc;q)_{k+1}(bd;q)_k}{(bxq;q)_{k}^{2}
  (ac;q)_k(ad;q)_{k+1}}q^k-\frac{(x-c)^{2}}{(x-d)^{2}}\sum^{n}_{k=1}\frac{(ax;q)_k^{2}
  (bc;q)_k(bd;q)_{k+1}}{(bxq;q)_{k}^{2}
  (ac;q)_{k+1}(ad;q)_k}q^k\nonumber\\
&\qquad=\frac{(c-d)(1-ax)^{2}
  (1-bc)(1-bd)}{(a-b)(x-d)^{2}
(1-ac)(1-ad)}\bigg(1-\frac{(axq;q)_{n}^{2}
  (bcq;q)_{n}(bdq;q)_{n}}{(bxq;q)_{n}^{2}
  (acq;q)_{n}(adq;q)_{n}}\bigg).
\end{align*}    
\end{prop}
\vspace{8pt}
\begin{tl}[$f(x,y)=\sin(x-y),$  $g(x,y)=\sin(x-y)$]
 \begin{align*}
    &\sum^{n}_{k=1}\frac{\sin(x_k-d_k)^2\sin(a_k-b_k)\sin(a_k-c_k)}{\sin(x_k-a_k)^{2}
\sin(d_k-b_k)\sin(d_k-c_k)}\\&\qquad\qquad\times\prod_{i=0}^{k-1}\frac{\sin(x_i-b_i)^{2}
  \sin(a_i-c_i)\sin(a_i-d_i)}{\sin(x_i-a_i)^{2}
\sin(b_i-c_i)\sin(b_i-d_i)}\nonumber\\
&+\sum^{n}_{k=1}\frac{\sin(x_k-c_k)^{2}\sin(a_k-b_k)\sin(a_k-d_k)}{\sin(x_k-a_k)^{2}
\sin(c_k-b_k)\sin(c_k-d_k)}\\
&\qquad\qquad\times\prod_{i=0}^{k-1}\frac{\sin(x_i-b_i)^{2}
  \sin(a_i-c_i)\sin(a_i-d_i)}{\sin(x_i-a_i)^{2}
\sin(b_i-c_i)\sin(b_i-d_i)}\\
&=\frac{\sin(x_0-b_0)^{2}
  \sin(a_0-c_0)\sin(a_0-d_0)}{\sin(x_0-a_0)^{2}
\sin(b_0-c_0)\sin(b_0-d_0)}\\
&\qquad\qquad-\prod_{k=0}^n\frac{\sin(x_k-b_k)^{2}
  \sin(a_k-c_k)\sin(a_k-d_k)}{\sin(x_k-a_k)^{2}
\sin(b_k-c_k)\sin(b_k-d_k)}.\nonumber
\end{align*}
\end{tl}
\begin{tl}[$f(x,y)=\cos(x-y),$ $g(x,y)=\sin(x-y)$]
\begin{align*}
 &\sum^{n}_{k=1}\frac{\cos(x_k-d_k)^2\sin(a_k-b_k)\sin(a_k-c_k)}{\cos(x_k-a_k)^{2}
\sin(d_k-b_k)\sin(d_k-c_k)}\\
&\qquad\qquad\times\prod_{i=0}^{k-1}\frac{\cos(x_i-b_i)^{2}
  \sin(a_i-c_i)\sin(a_i-d_i)}{\cos(x_i-a_i)^{2}
\sin(b_i-c_i)\sin(b_i-d_i)}\nonumber\\
&+\sum^{n}_{k=1}\frac{\cos(x_k-c_k)^{2}\sin(a_k-b_k)\sin(a_k-d_k)}{\cos(x_k-a_k)^{2}
\sin(c_k-b_k)\sin(c_k-d_k)}\\
&\qquad\qquad\times\prod_{i=0}^{k-1}\frac{\cos(x_i-b_i)^{2}
  \sin(a_i-c_i)\sin(a_i-d_i)}{\cos(x_i-a_i)^{2}
\sin(b_i-c_i)\sin(b_i-d_i)}\\
&=\frac{\cos(x_0-b_0)^{2}
  \sin(a_0-c_0)\sin(a_0-d_0)}{\cos(x_0-a_0)^{2}
\sin(b_0-c_0)\sin(b_0-d_0)}\\&\qquad\qquad-\prod_{k=0}^n\frac{\cos(x_k-b_k)^{2}
  \sin(a_k-c_k)\sin(a_k-d_k)}{\cos(x_k-a_k)^{2}
\sin(b_k-c_k)\sin(b_k-d_k)}.\nonumber
\end{align*}
\end{tl}
The next identity is an elliptic analogue of \eqref{idnew-1}. 
\begin{tl}[$f(x,y)=g(x,y)=y\theta(xy,x/y;q)$]
 \begin{align}
    &\sum^{n}_{k=1}\frac{\theta\left(x_kd_k,  d_k/x_k ; q\right)^{2}
  \theta\left( a_kb_k, a_k/ b_k,a_kc_k, a_k/ c_k; q\right)}{\theta\left(x_ka_k,  a_k/x_k ; q\right)^{2}
  \theta\left(d_kb_k, d_k/b_k, c_kd_k, d_k/c_k; q\right)}\nonumber\\
  &\qquad\qquad\times\prod_{i=0}^{k-1}\frac{\theta\left(x_ib_i, b_i/x_i ; q\right)^{2}
  \theta\left(a_ic_i, a_i/ c_i, a_id_i, a_i/ d_i; q\right)}{\theta\left(x_ia_i, a_i/x_i ; q\right)^{2}
  \theta\left(b_ic_i, b_i/ c_i, b_id_i, b_i/ d_i; q\right)}\nonumber\\
&+\sum^{n}_{k=1}\frac{\theta\left(x_kc_k,c_k/ x_k ; q\right)^{2}
  \theta\left(a_kb_k, a_k/ b_k, a_kd_k, a_k/ d_k; q\right)}{\theta\left(x_ka_k, a_k/x_k ; q\right)^{2}
  \theta\left(b_kc_k, c_k/b_k, c_kd_k, c_k/ d_k; q\right)}\nonumber\\
  &\qquad\qquad\times\prod_{i=0}^{k-1}\frac{\theta\left(x_ib_i, b_i/x_i ; q\right)^{2}
  \theta\left(a_ic_i, a_i/ c_i, a_id_i, a_i/ d_i; q\right)}{\theta\left(x_ia_i, a_i/x_i ; q\right)^{2}
  \theta\left(b_ic_i, b_i/ c_i, b_id_i, b_i/ d_i; q\right)}\label{idnew-123}\\
&=\frac{\theta\left(x_0b_0, b_0/x_0 ; q\right)^{2}
  \theta\left(a_0c_0, a_0/ c_0, a_0d_0, a_0/ d_0; q\right)}{\theta\left(x_0a_0, a_0/x_0 ; q\right)^{2}
  \theta\left(b_0c_0, b_0/ c_0, b_0d_0, b_0/ d_0; q\right)}\nonumber\\
  &\qquad\qquad-\prod_{k=0}^n\frac{\theta\left(x_kb_k, b_k/ x_k ; q\right)^{2}
  \theta\left(a_kc_k, a_k/ c_k, a_kd_k, a_k/ d_k; q\right)}{\theta\left(x_ka_k, a_k/ x_k ; q\right)^{2}
  \theta\left(b_kc_k, b_k/ c_k, b_kd_k, b_k/ d_k; q\right)}.\nonumber
\end{align} 
\end{tl}
\pf The consequence follows after some simplifications by the relation
\begin{align*}\theta(z;q)&=\theta(q/z;q)=-z\theta(1/z;q).
 \end{align*}\qed

 With elliptic hypergeometric series in mind, we  need to set in \eqref{idnew-123} that $c_i=c
,d_i=d~(c\neq d)$ and $$x_i=x,a_i=ap^i,b_i=bp^i$$ for seeking such a kind of  transformation formulas.
\begin{prop}Let $(x;p,q)_n$ be defined by \eqref{ellptic-shift}. Then for any integer $n\geq 0$, we have
 \begin{align}
  \chi(c,d)&\sum^{n}_{k=1}\frac{
  \theta\left( abp^{2k},acp^k, ap^k/ c; q\right)}{\theta\left(xap^k,  ap^k/x ; q\right)^{2}
  \theta\left(bdp^k, dp^{-k}/b; q\right)}\nonumber\\
  &\qquad\qquad\times\frac{\left(xb, b/x ; p,q\right)_k^{2}
  \left(ac, a/ c, ad, a/ d; p,q\right)_k}{\left(xa, a/x ; p,q\right)_k^{2}
  \left(bc, b/ c, bd, b/ d; p,q\right)_k}\nonumber\\
+\chi(d,c)&\sum^{n}_{k=1}\frac{
  \theta\left(abp^{2k}, adp^k, ap^k/ d; q\right)}{\theta\left(xap^k, ap^k/x ; q\right)^{2}
  \theta\left(bcp^k, cp^{-k}/b; q\right)}\nonumber\\
  &\qquad\qquad\times\frac{\left(xb, b/x ; p,q\right)_k^{2}
  \left(ac, a/ c, ad, a/ d; p,q\right)_k}{\left(xa, a/x ; p,q\right)_k^{2}
  \left(bc, b/ c, bd, b/ d; p,q\right)_k}\label{idnew-888}
 \end{align}
 \begin{align}
&=\frac{\theta\left(xb, b/x ; q\right)^{2}
  \theta\left(ac, a/ c, ad, a/ d; q\right)}{\theta\left(xa, a/x ; q\right)^{2}
  \theta\left(bc, b/ c, bd, b/ d; q\right)}\nonumber\\
  &\qquad\qquad-\frac{\left(xb, b/x ; p,q\right)_{n+1}^{2}
  \left(ac, a/ c, ad, a/ d; p,q\right)_{n+1}}{\left(xa, a/x ; p,q\right)_{n+1}^{2}
  \left(bc, b/ c, bd, b/ d; p,q\right)_{n+1}},\nonumber
\end{align}
where the function 
$$\chi(c,d):=\frac{\theta\left(xd,  d/x ; q\right)^{2}\theta\left(a/b; q\right)}{
  \theta\left(cd, d/c; q\right)}.$$
\end{prop}
\pf Since all computations are routine, we thus leave to the interested reader.\qed

\vspace{10pt}
\noindent{\bf Conflict of Interest Statement.} The authors declare that there is no competing interests relating to this work. 

\vspace{10pt}
\noindent{\bf Funding Declaration.} This work is supported by the National Natural Science Foundation of China [Grant
No. 12471315].


\begin{thebibliography}{99}
\bibliographystyle{amsplain}
\bibitem{10}   G. Gasper,   M. Rahman,  \emph{Basic Hypergeometric Series}, 2nd edition, Cambridge University Press, 2004.
    \bibitem{gasper}    G. Gasper, M.  Schlosser, \emph{Summation, transformation, and expansion formulas
for multibasic theta hypergeometric series}, Adv. Stud. Contemp. Math. \textbf{11 }(2005), 67--84.
\bibitem{koornwinder}T.  H.  Koornwinder, \emph{On the equivalence of two fundamental theta identities}, Anal. Appl. (Singap.) \textbf{12} (2014),  711--725.
\bibitem{knopp}K. Knopp, \textit{Theory and Application of Infinite Series}, Dover Books on Mathematics, Dover Publications, Mineola, NY, 1990.
\bibitem{fgsummation}X. R. Ma,\textit{ The $(f, g)$-inversion formula and its applications: the $(f, g)$-summation formula,} Adv. Appl.
Math. \textbf{38(2)}  (2007), 227--257.
\bibitem{wangjinpaper}J. Wang, \emph{A new elliptic interpolation formula via the $(f,g)$-inversion}, Proc. Amer. Math. Soc. \textbf{148}  (2020), 3457--3471.

 \bibitem{warnaar}S. O. Warnaar, \emph{Summation and transformation formulas for elliptic hypergeometric series}, Constr. Approx. \textbf{18} (2002), 479--502.
\bibitem{xuxrma} J. N. Xu, X. R. Ma, \emph{On two elementary functional identities}, Bull. Malays. Math. Sci. Soc. \textbf{48} (2025), 181.
 \end{thebibliography}
\end{document}